\documentclass[11pt, reqno]{amsart}

\usepackage{amsfonts,amssymb,latexsym,amsmath, amsxtra}
\usepackage{verbatim}
\usepackage{amsthm}
\usepackage{amsmath,mathtools}
\usepackage{amscd,enumitem}
\usepackage{verbatim}
\usepackage{eurosym}
\usepackage{float}
\usepackage{color}
\usepackage{xcolor}
\usepackage{dcolumn}
\usepackage[mathscr]{eucal}
\usepackage[all]{xy}
\usepackage{bbm}
\usepackage{tikz}
\usetikzlibrary{arrows}
\usepackage[OT2,T1]{fontenc}
\DeclareSymbolFont{cyrletters}{OT2}{wncyr}{m}{n}
\DeclareMathSymbol{\Sha}{\mathalpha}{cyrletters}{"58}

\theoremstyle{plain}
\newtheorem{theorem}{Theorem}[section]
\newtheorem*{theorem*}{Theorem}

\newtheorem{lemma}[theorem]{Lemma}

\newtheorem*{conjecture*}{Conjecture}
\newtheorem*{question*}{Question}

\theoremstyle{definition}

\theoremstyle{remark}

\newtheorem*{remark*}{Remark}
\newtheorem*{remarks*}{\bf Remarks}

\newtheorem*{conj*}{Conjecture}

\theoremstyle{definition}
\theoremstyle{plain}
\newtheorem{cor}[theorem]{Corollary}

\newtheorem{prop}[theorem]{Proposition}

\numberwithin{equation}{section}

\newcommand{\R}{\mathbb R}
\newcommand{\N}{\mathbb N}

\newcommand{\C}{\mathbb C}

\makeatletter
\newcommand*{\rom}[1]{\expandafter\@slowromancap\romannumeral #1@}
\makeatother

\def\({\left(}
\def\){\right)}

\DeclareMathOperator\Log{Log}

\numberwithin{equation}{section}
\numberwithin{theorem}{section}

\def\lp{\left(}
\def\rp{\right)}

\usepackage{centernot}

\title[A note on hook length equidistribution]{A note on hook length equidistribution on arithmetic progressions}

\author{Joshua Males}
\address{School of Mathematics, University of Bristol, Bristol, BS8 1TW, UK, and the Heilbronn Institute for Mathematical Research, Bristol, UK.}
\email{joshua.males@bristol.ac.uk}

\begin{document}
	
	\begin{abstract}
In a recent paper, Bringmann, Craig, Ono, and the author showed that the number of $t$-hooks ($t\geq2$) among all partitions of $n$ is not always asymptotically equidistributed on congruence classes $a \pmod{b}$. In this short note, we clarify the situation of $t=1$, i.e.\@ all hook lengths, and show that this case does give asymptotic equidistribution, closing the story of the distribution properties of $t$-hooks on congruence classes.
	\end{abstract}
	
	
	
\maketitle
	
\section{Introduction and statement of results}

A partition of a positive integer $n$ is a non-increasing list $\lambda = (\lambda_1, \dots, \lambda_s)$ such that $\sum_j \lambda_j =n$. The theory of partitions is a well-studied area with an abundance of beautiful results. One of the most famous is the asymptotic formula given by Hardy and Ramanujan around a century ago, which says that as $n \to \infty$ we have that
\begin{align*}
	p(n) \sim \frac{1}{4n\sqrt{3}} e^{\pi \sqrt{\frac{2n}{3}}}.
\end{align*}
In order to prove their result, Hardy and Ramanujan introduced the so-called Circle Method, a method which has seen a vast amount of applications in the field over the last 100 years. In a series of papers, Wright \cite{Wright1,Wright2} gave a more flexible variant of the Circle Method (see Section \ref{Sec: WCM}) which is widely applicable in obtaining the main term asymptotic for many exponentially growing sequences, but trades this flexibility for the loss of information in that it cannot be extended to obtain exact formulae for the coefficients.

In more recent times, Wright's Circle Method has been used in many papers in the literature the most prominent example that is relevant to the present paper is in \cite{BCMO} which was concerned with so-called $t$-hooks of partitions. Recall that a partition $\lambda =  (\lambda_1, \dots, \lambda_s)$ can also be viewed by a Ferrers--Young diagram 
$$
\begin{matrix}
	\bullet & \bullet & \bullet & \dots & \bullet & \leftarrow & \lambda_1 \text{ many nodes}\\
	\bullet & \bullet & \dots & \bullet & & \leftarrow & \lambda_2 \text{ many nodes}\\
	\vdots & \vdots & \vdots & &  &  \\
	\bullet & \dots & \bullet & & & \leftarrow & \lambda_s \text{ many nodes},
\end{matrix}
$$
and each node in the diagram has a so-called hook length. The node in row $k$ and column $j$ has hook length
$h(k,j):=(\lambda_k-k)+(\lambda'_j-j)+1,$ where $\lambda'_j$ is the number of nodes in column $j$.

In \cite{BCMO}, Bringmann, Craig, Ono, and the author studied the $t$-hooks of partitions (that is, hook lengths divisible by $t\geq2$) and their distribution properties over arithmetic progressions. However, the case $t=1$ (i.e.\@ all hook lengths) was not studied owing to the somewhat different formulation of its generating function. In particular, in a beautiful paper, Han \cite{Han} gave the following generating function for the number of hook lengths of length $\ell$ among all partitions of $n$.

\begin{align*}
	H(\zeta;q) \coloneqq \sum_{m,n \geq 0} h_\ell(m,n) \zeta^m q^n = \sum_{ \lambda \in \mathcal{P} } q^{|\lambda|} \zeta^{\# \{ h \in \mathcal{H}(\lambda), h=\ell \}} = \prod_{n \geq 1} \frac{ \left(1+(\zeta-1)q^{\ell n}\right)^\ell}{1-q^n}.
\end{align*}

We remark that the numerator is not amenable to classical methods of asymptotic analysis due to the presence of the factor $(\zeta-1)$ in the numerator. However, in a recent paper \cite{BCM} Bringmann, Craig, and the author showed that we may in fact determine the asymptotic behaviour of products of the shape
\begin{align*}
	\prod_{n \geq 1} f(q^n)
\end{align*}
where $f$ are polynomials with complex-valued coefficients carrying certain analytic properties. We follow those ideas in this note.

Let $h(a,b;n)$ count the number of hook lengths of length $\ell \equiv a \pmod{b}$ among all partitions of $n$. In this note we follow the techniques of \cite{BCM} alongside Wright's Circle Method to determine our main theorem, showing that $h(a,b;n)$ are asymptotically equidistributed as $n \to \infty$, thereby finishing the story when taken alongside the results in \cite{BCMO}.

\begin{theorem}\label{Thm: main}
	As $n \to \infty$ we have
	\begin{align*}
		h(a,b;n) = \frac{p(n)}{b} \left(1 + O\left(n^{-\frac{1}{2}}\right)\right)= \frac{1}{4b n \sqrt{3}} e^{\pi \sqrt{\frac{2n}{3}}} \left(1 + O\left(n^{-\frac{1}{2}}\right)\right).
	\end{align*}
	
\end{theorem}

It is worth pointing out that our results are not surprising to experts in the field, and that similar analysis for $\zeta \in \R$ was conducted by Griffin, Ono, and Tsai in \cite[Proposition 2.1]{GOT} in their investigations into the limiting distributions of hook lengths. In that paper, the authors used the saddle-point method, Taylor series expansions, and technical estimates. In using the Euler-Maclaurin summation formula in the current paper, it is possible to obtain asymptotic formulae of arbitrary precision.

\subsection*{Acknowledgements}
The author thanks Will Craig for pointing out a typo in the first version of this note.

\section{Preliminaries}

\subsection{Asymptotics of infinite $q$-products}

 We recall the generating function for partitions is
\begin{align*}
	P(q) \coloneqq \frac{1}{(q;q)_\infty},
\end{align*}
where we set $(a)_j =(a;q)_j \coloneqq  \prod_{\ell=0}^{j-1}(1-aq^\ell)$ for $j\in\N_0\cup \{\infty\}$. Its representation as (essentially) the inverse of the Dedekind $\eta$-function famously gives the classical asymptotic behaviour (see e.g.\@ \cite[equation (2.7)]{BCMO} with $k=1,h=0$ and shifting $z\mapsto \frac{z}{2\pi}$)
\begin{align} \label{equ: asymptotic P major arc}
	\left(e^{-z}; e^{-z}\right)_\infty^{-1} = \left(\frac{z}{2\pi}\right)^{\frac 12} e^{\frac{\pi}{12}\left(\frac{2\pi}{z}-\frac{z}{2\pi}\right)} \left(e^{-\frac{4\pi^2}{z}},e^{-\frac{4\pi^2}{z}} \right)_\infty^{-1},
\end{align}
for $z\in\C$ with $\operatorname{Re}(z)>0$. This will provide us the asymptotic main term on the major arc in our application of Wright's Circle Method.

For the minor arc, we recall \cite[Lemma 3.5]{BD}.
\begin{lemma} \label{Lem: asymptotic P minor arc}
	Let $M>0$ be a fixed constant. Assume that $\tau =u+iv \in \mathbb{H}$, with $Mv\leq |u| \leq \frac{1}{2}$ for $u>0$ and $v\to 0$. We have that
	\begin{align*} 
		|P(q)| \ll \sqrt{v} \exp \left[\frac 1v \left(\frac{\pi}{12} -\frac{1}{2\pi}\left(1-\frac{1}{\sqrt{1+M^2}}\right)\right)\right].
	\end{align*}
\end{lemma}

\section{Asymptotic techniques} \label{Asymptotic techniques}
Here we recall two techniques used in the proof of our main theorem, each of which is used extensively in the literature and related works.

\subsection{A variation of Euler--Maclaurin summation}
In essence, for the purposes of this paper, the Euler-Maclaurin summation formula can be seen as a way to relate sums (which are in general difficult to estimate asymptotically) to integrals, which are much more amenable to other techniques. We use a version which appeared several years ago in the literature, and to do so require some notation. We say that a function $f$ is of sufficient decay in a domain $D\subset\C$ if there exists some $\varepsilon > 0$ such that $f(w) \ll w^{-1-\varepsilon}$ as $|w| \to \infty$ in $D$. We quote Theorem 1.2 of \cite{BJM} which follows from the Euler--Maclaurin summation formula.

\begin{theorem}\label{Theorem:EulerMaclaurin1DShifted}
	Suppose that $0\le \theta < \frac{\pi}{2}$ and let
	$D_\theta := \{ re^{i\alpha} : r\ge0 \mbox{ and } |\alpha|\le \theta  \}$.
	Let $f:\C\rightarrow\C$ be holomorphic in a domain containing
	$D_\theta$, so that in particular $f$ is holomorphic at the origin, and
	assume that $f$ and all of its derivatives are of sufficient decay.
	Then for $a\in\mathbb{R}$ and $N\in\N_0$,
	\begin{equation*}
		\sum_{m\geq0}f((m+a)w) = \frac{I_{F}}{w} - \sum_{n=0}^{N-1} \frac{B_{n+1}(a) f^{(n)}(0)}{(n+1)!}w^n + O_N\left(w^N\right),
	\end{equation*}
	uniformly, as $w\rightarrow0$ in $D_\theta$. Here $I_{F}:=\int_0^\infty f(x) dx$.
\end{theorem}

\subsection{Wright's Circle Method}\label{Sec: WCM}

Wright's Circle Method is a more flexible variant of the original method of Hardy and Ramanujan, and has seen many applications in recent years. Here we recall a variant proved in \cite{BCMO} which is particularly easy-to-use and applicable to a wide range of generating functions.

\begin{prop} \label{WrightCircleMethod}
	Suppose that $F(q)$ is analytic for $q = e^{-z}$ where $z=x+iy \in \C$ satisfies $x > 0$ and $|y| < \pi$, and suppose that $F(q)$ has an expansion $F(q) = \sum_{n=0}^\infty c(n) q^n$ near 1. Let $c,N,M>0$ be fixed constants. Consider the following hypotheses:
	
	\begin{enumerate}[leftmargin=*]
		\item[\rm(1)] As $z\to 0$ in the bounded cone $|y|\le Mx$ (major arc), we have
		\begin{align*}
			F(e^{-z}) = z^{B} e^{\frac{A}{z}} \left( \sum_{j=0}^{N-1} \alpha_j z^j + O \left(|z|^N\right) \right),
		\end{align*}
		where $\alpha_s \in \C$, $A\in \R^+$, and $B \in \R$. 
		
		\item[\rm(2)] As $z\to0$ in the bounded cone $Mx\le|y| < \pi$ (minor arc), we have 
		\begin{align*}
			\lvert	F(e^{-z}) \rvert \ll e^{\frac{1}{\mathrm{Re}(z)}(A - \kappa)}.
		\end{align*}
		for some $\kappa\in \R^+$.
	\end{enumerate}
	If  {\rm(1)} and {\rm(2)} hold, then as $n \to \infty$ we have for any $N\in \R^+$ 
	\begin{align*}
		c(n) = n^{\frac{1}{4}(- 2B -3)}e^{2\sqrt{An}} \lp \sum\limits_{r=0}^{N-1} p_r n^{-\frac{r}{2}} + O\left(n^{-\frac N2}\right) \rp,
	\end{align*}
	where $p_r := \sum\limits_{j=0}^r \alpha_j c_{j,r-j}$ and $c_{j,r} := \dfrac{(-\frac{1}{4\sqrt{A}})^r \sqrt{A}^{j + B + \frac 12}}{2\sqrt{\pi}} \dfrac{\Gamma(j + B + \frac 32 + r)}{r! \Gamma(j + B + \frac 32 - r)}$. 
\end{prop}

\section{Proof of Theorem \ref{Thm: main}}

\begin{proof}[Proof of Theorem \ref{Thm: main}]
	
Let $h(a,b;n)$ count the number of hook lengths of length $\ell \equiv a \pmod{b}$. Then we have that
\begin{align}\label{eqn: splitting}
	\sum_{n \geq 0} h(a,b;n) q^n = \frac{1}{b} \sum_{k=0}^{b-1} \zeta_b^{-ak} H(\zeta_b^k;q).
\end{align}
We aim to show that the $k=0$ term on the right-hand side dominates the asymptotic behaviour as $w \to 0$ in $D_\theta$. Since this term is independent of $b$ we then immediately obtain asymptotic equidistribution.

First note that the $k=0$ term gives a contribution of
\begin{align*}
	H(1;q) = P(q) = (q;q)_\infty^{-1}.
\end{align*}
This is the generating function for partitions, and is well-studied in the literature. In particular, we have by \eqref{equ: asymptotic P major arc} and Lemma \ref{Lem: asymptotic P minor arc} that $P(q)$ satisfies the conditions of Proposition \ref{WrightCircleMethod}.

The idea now is simple. We just need to show that the numerator of each $k\neq0$ term in \eqref{eqn: splitting} gives an exponentially small contribution. Equivalently, we show that the main term in the asymptotic of logarithm of the numerator has a negative real part.

First, let $\xi$ be a root of unity not equal to $\pm 1$. Then we first want to determine the asymptotic behaviour of
\begin{align*}
\mathcal{F}_\xi(q) \coloneqq	\prod_{n \geq 1} \left(1+(\xi-1)q^{\ell n}\right)^\ell = \prod_{n \geq 1} P_\xi(q^n)^\ell
\end{align*} 
with $P_\xi(x) \coloneqq 1+ (\xi-1)x^\ell$.

We note that $P(x) \not\in (-\infty,0]$ for $x \in [0,1]$, and so we avoid the branch cut of the complex logarithm in the following. We take logarithms (using the principal branch throughout) to obtain
\begin{align*}
	G_\xi(q) \coloneqq \ell \sum_{n \geq 1} \Log\left( P_\xi(q^n) \right). 
\end{align*}
We now estimate the sum for $q = e^{-w}$ as $w \to 0$ in $D_\theta$ using Euler-Maclaurin summation. Write 
\begin{align*}
	G_\xi(e^{-w}) = \ell \sum_{n \geq 1} f_\xi(nw),
\end{align*}
where 
\begin{align*}
	f_\xi(z) \coloneqq \Log(P(e^{-z})).
\end{align*}
Applying Theorem \ref{Theorem:EulerMaclaurin1DShifted} then gives that
\begin{align*}
	G_\xi(e^{-w}) &= \frac{\ell I_{f_\xi}}{w} -  \ell \sum_{n=0}^{N-1} \frac{B_{n+1}(1) f_\xi^{(n)}(0)}{(n+1)!}w^n + O_N\left(|w|^N\right)\\
	&= \frac{\ell I_{f_\xi}}{w} - \frac{\ell  \Log(\xi)}{2} + O(|w|).
\end{align*}
Thus we see that as $w \to 0$ in $D_\theta$ we have
\begin{align*}
	\mathcal{F}_\xi (e^{-w}) \sim \xi^{\frac{\ell}{2}} \exp\left( \ell I_{f_\xi}\right).
\end{align*}
Now we need only check that 
\begin{align*}
\Re\left(	I_{f_\xi} \right) < 0.
\end{align*}
This is obvious, since
\begin{align*}
	\Re\left( I_{f_\xi}  \right) = 	\Re\left(  \int_0^\infty \Log(1+(\xi-1)e^{-\ell x})  dx  \right) = \int_0^\infty 	\Re\left(  \Log(1+(\xi-1)e^{-\ell x})   \right) dx <0,
\end{align*}
which can be seen via the Mercator series representation of the complex logarithm $\Log(1+u) = \sum_{n \geq 1} \frac{(-1)^{n+1}u^n}{n}$ for $|u|<1$.
Therefore on both the major and minor arcs (since $w \to 0$ in $D_\theta$), the contribution of $\mathcal{F}_\xi$ is dominated by the $k=0$ term in \eqref{eqn: splitting}.

Now, when $\xi =-1$ we have the contribution
\begin{align*}
\mathcal{G}(w) \coloneqq	\prod_{n \geq 1} \left(1-2e^{-\ell w}\right)^\ell.
\end{align*}
Taking absolute values and then the logarithm gives
\begin{align*}
	\Log(|\mathcal{G}(w)|) = \ell \sum_{n \geq 1} \Log\left( \lvert 1-2e^{-\ell w} \rvert \right).
\end{align*}
Then again applying Theorem \ref{Theorem:EulerMaclaurin1DShifted} the main term asymptotic contribution is
\begin{align*}
\Log(|\mathcal{G}(w)|) \sim	\frac{\ell I_{g}}{w}
\end{align*}
with 
\begin{align*}
	g(x) \coloneqq \log\left( \lvert 1-2x^{-\ell} \rvert \right).
\end{align*}
Now, we have
\begin{align*}
	I_g = \int_0^\infty  \log\left( \lvert 1-2e^{-\ell x} \rvert \right) dx <0
\end{align*}
as desired. After exponentiating, we similarly obtain that the asymptotic contribution of the $\zeta = -1$ term is exponentially dominated by the $k=0$ term in \eqref{eqn: splitting}.

Then since the $k=0$ term (i.e.\@ $P(q)$) gives an exponentially dominant asymptotic contribution to \eqref{eqn: splitting} as $w \to 0$ on both the major and minor arcs. We now use Proposition \ref{WrightCircleMethod} with the parameters $A=\frac{\pi^2}{6}$, $B= \frac{1}{2}$ and $\alpha_0 = \frac{1}{\sqrt{2\pi}}$  to obtain that as $n \to \infty$ we have
\begin{align*}
	h(a,b;n) = \frac{1}{4bn\sqrt{3}} e^{\pi \sqrt{\frac{2n}{3}}} \left(1+O\left(n^{-\frac{1}{2}}\right)\right) = \frac{p(n)}{b} \left(1+O\left(n^{-\frac{1}{2}}\right)\right)
\end{align*}
as desired.
\end{proof}

\section{Further results}

Here we discuss other results for $h(a,b;n)$ which are simple to conclude given Theorem \ref{Thm: main}. Firstly, we immediately obtain that $h(a,b;n)$ automatically satisfy all of the higher-order Tur\'{a}n inequalities via \cite[Theorem 3]{GORZ}  and the higher order Laguerre inequalities via \cite[Theorem 1.4]{Wagner}. Each essentially arise since the main term asymptotic is controlled by a modular form (up to a power of $q$), and the results follow easily.

\begin{cor}
	For large enough $n$, the coefficients $h(a,b;n)$ satisfy all higher-order Tur\'{a}n inequalities and all higher order Laguerre inequalities.
\end{cor}

We also immediately obtain the following corollary via  \cite[Corollaries 3.2 and 3.3]{CCM} .

\begin{cor}
	For large enough $n_1$ and $n_2$, we have that
	\begin{align*}
	h(a,b;n_1)h(a,b;n_2) >h(a,b;n_1+n_2).
	\end{align*}
\end{cor}

Numerical experiments also suggest that for each fixed $n$ the roots of the polynomials $p$ defined by $H_\ell(\zeta,q)= \sum_{n \geq 0} p_\ell(\zeta) q^n$ are all negative for any $\ell$, and so it may give an interesting new example of the slightly stronger class of Laguerre-Polya functions (called functions of Type I) of Wagner \cite{Wagner}. It also appears that these polynomials are unimodal. It would be very interesting to understand these properties in further detail, and we leave this as a question to the interested reader.

\end{document}